\documentclass[11pt]{amsart}

\usepackage[a4paper,hmargin=3cm,vmargin=3cm]{geometry}
\usepackage{amsfonts,amssymb,amscd,amstext}
\usepackage{graphicx}
\usepackage[dvips]{epsfig}
\usepackage[latin1]{inputenc}

\usepackage{fancyhdr}
\usepackage{times}

\usepackage{enumerate}
\usepackage{titlesec}
\usepackage{mathrsfs}
\usepackage{stmaryrd}

\def\R{\mathbb{R}}

\def\C{\mathbb{C}}

\newcommand{\ben}{\begin{enumerate}}
\newcommand{\bit}{\begin{itemize}}
\newcommand{\een}{\end{enumerate}}
\newcommand{\eit}{\end{itemize}}

\newcommand{\ed}{\end{document}}

\def\cA{\mathcal{A}}
\def\cU{\mathcal{U}}

\def\cD{\mathcal{D}}
\def\cR{\mathcal{R}}

\def\cW{\mathcal{W}}

\def\cV{\mathcal{V}}

\def\cH{\mathcal{H}}
\def\cL{\mathcal{L}}

\def\cN{\mathcal{N}}
\def\cK{\mathcal{K}}

\let\8=\infty \let\0=\emptyset

\let\landa=\lambda
\let\alfa=\alpha

\let\parc=\partial

\def\ep{\varepsilon}

\def\landa{\lambda}

\def\flecha{\rightarrow}
\def\esiz{\langle}
\def\esde{\rangle}

\def\S{\Sigma}

\def\cte.{\mathop{\rm cte.}\nolimits}

\def\R{\mathbb{R}}

\def\C{\mathbb{C}}

\def\H{\mathbb{H}}
\def\S{\mathbb{S}}

\def\X{\mathfrak{X}}

\newfont{\bb}{msbm10 at 12pt}

\titleformat{\section}
{\filcenter\bfseries\large} {\thesection{.}}{0.2cm}{}
\titleformat{\subsection}[runin]
{\bfseries} {\thesubsection{.}}{0.15cm}{}[.]
\titleformat{\subsubsection}[runin]
{\em}{\thesubsubsection{.}}{0.15cm}{}[.]

\usepackage[up,bf]{caption}

\newtheorem{theorem}{Theorem}[section]
\newtheorem{lemma}[theorem]{Lemma}

\newtheorem{remark}[theorem]{Remark}
\newtheorem{corollary}[theorem]{Corollary}

\theoremstyle{definition}

\numberwithin{equation}{section}
\numberwithin{figure}{section}

\usepackage{color}

\begin{document}
\fancyhead[LO]{A quasiconformal Hopf soap bubble theorem}
\fancyhead[RE]{José A. Gálvez, Pablo Mira, Marcos P. Tassi}
\fancyhead[RO,LE]{\thepage}

\thispagestyle{empty}

\begin{center}
{\bf \LARGE A quasiconformal Hopf soap bubble theorem}
\vspace*{5mm}

\hspace{0.2cm} {\Large José A. Gálvez, Pablo Mira, Marcos P. Tassi}
\end{center}

\footnote[0]{
\mbox{} \hspace{0.25cm} \emph{Keywords}: Constant mean curvature, quasiconformal Gauss map, Hopf theorem, index theory, uniqueness, elliptic equations, Bers-Nirenberg representation, Alexandrov conjecture}



\vspace*{7mm}

\begin{quote}
{\small
\noindent {\bf Abstract}\hspace*{0.1cm}
We show that any compact surface of genus zero in $\R^3$ that satisfies a quasiconformal inequality between its principal curvatures is a round sphere. This solves an old open problem by H. Hopf, and gives a spherical version of Simon's quasiconformal Bernstein theorem. The result generalizes, among others, Hopf's theorem for constant mean curvature spheres, the classification of round spheres as the only compact elliptic Weingarten surfaces of genus zero, and the uniqueness theorem for ovaloids by Han, Nadirashvili and Yuan. The proof relies on the Bers-Nirenberg representation of solutions to linear elliptic equations with discontinuous coefficients.



\vspace*{0.1cm}

}
\end{quote}


\section{Introduction}

Hopf's soap bubble theorem is one of the fundamental theorems of constant mean curvature (CMC) theory. It states that any CMC sphere immersed in $\R^3$ is a round sphere \cite{Ho}; here, by a \emph{sphere} or \emph{immersed sphere}, we mean a smooth, compact orientable surface of genus zero immersed in $\R^3$. Hopf's soap bubble theorem has been extended to several important geometric situations, such as elliptic Weingarten spheres in $\R^3$ (\cite{Ho,Ch,HW,Br,GM3}), or CMC spheres in Riemannian homogeneous $3$-manifolds (\cite{Al,Ca,Ch2,AR1,AR2,DM,Me,MMPR1,MMPR2}, the final classification being obtained in \cite{MMPR2}).

In this paper we prove a quasiconformal extension of Hopf's theorem. In it, we replace the CMC hypothesis by the general quasiconformal inequality 
\begin{equation}\label{cua0}
(H-c)^2 \leq \mu (H^2-K),
\end{equation}
where $\mu,c$ are constants, with $\mu<1$, and $H,K$ denote the mean and Gaussian curvatures of the surface. Obviously, when $\mu=0$ we obtain the CMC condition $H=c$, but the case $\mu\in (0,1)$ models a much more general class of surfaces. Condition \eqref{cua0} has its origins in some classical problems of surface theory considered, among others, by Alexandrov, Hopf, Pogorelov, Osserman, Simon or Schoen, that we explain next.

When $c=0$, inequality \eqref{cua0} corresponds to the property that the Gauss map of the surface is quasiconformal, and this defines a well-known class of surfaces. They were classically introduced by Finn \cite{Fi} (for the case of graphs), and by Osserman \cite{Os}, who called them \emph{quasiminimal surfaces}. The problem of determining which properties of minimal surfaces remain true in the quasiconformal setting of \eqref{cua0} has been deeply studied, see e.g. \cite{Fi,Os,Sim,SS} and also \cite{GT}. Of special importance is the following quasiconformal Bernstein theorem, by L. Simon \cite{Sim}: \emph{planes are the only entire graphs with quasiconformal Gauss map}.

Aligned with the classical \emph{quasiminimal} terminology, we define a \emph{quasi-CMC surface} as a smooth ($C^{\8}$) immersed surface in $\R^3$ satisfying \eqref{cua0} for some values of $c,\mu$; here the term \emph{quasi} refers to the quasiconformal nature of inequality \eqref{cua0}. Obviously, quasiminimal surfaces are quasi-CMC.

An equivalent way of writing \eqref{cua0} is 
 \begin{equation}\label{cuasi}
(\kappa_1-c)^2 +(\kappa_2-c)^2 \leq 2\Lambda\, (\kappa_1-c)(\kappa_2-c),
\end{equation}
where $\Lambda\leq -1$, $c\in \R$ and $\kappa_1\geq \kappa_2$ are the principal curvatures of the surface. Here, $\Lambda=\frac{\mu+1}{\mu-1}$. In this equivalent form, \eqref{cuasi} appears as the uniformly elliptic case of Alexandrov's inequality \eqref{ecur1} (see Section \ref{sec:2}), and as an extension of Hopf's classical \emph{cusp property} {\bf (H)} that will be explained below. A useful visual interpretation of \eqref{cuasi} in terms of the \emph{curvature diagram} of the surface is given in Figure \ref{fig:wedges1}. 

\begin{figure}[htbp]
    \centering
    \includegraphics[width=0.9\textwidth]{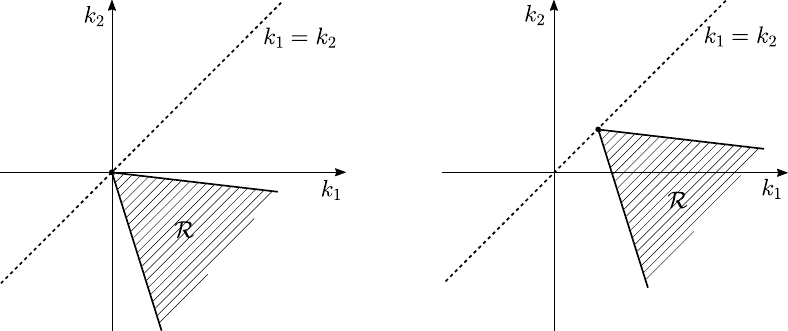}
     \caption{Curvature diagram regions for quasiminimal and quasi-CMC surfaces. A surface $\Sigma$ is quasiminimal (resp. quasi-CMC) if its curvature diagram $(\kappa_1,\kappa_2)(\Sigma)\subset \R^2$, $\kappa_1\geq \kappa_2$, lies in the wedge region $\cR$ in the left (resp. right). } 
\label{fig:wedges1}
\end{figure}

In general, the class of surfaces satisfying \eqref{cua0} is vast, and should not be expected to show many similarities with CMC surface theory. For instance, quasi-CMC surfaces do not satisfy the maximum principle, and Alexandrov's classical theorem for compact embedded CMC surfaces does not hold in the quasiconformal setting. Indeed, any thin enough tube around an arbitrary smooth Jordan curve of $\R^3$ is an embedded quasi-CMC torus.

In contrast, Theorem \ref{th:main1} below shows that Hopf's soap bubble theorem admits a general quasiconformal extension. It can be seen as a spherical version of Simon's quasiconformal Bernstein theorem \cite{Sim} stated above.

\begin{theorem}\label{th:main1}
Any quasi-CMC sphere immersed in $\R^3$ is a round sphere.
\end{theorem}

\begin{remark}\label{rem:1}
The proof of Theorem \ref{th:main1} also holds if we only assume that \eqref{cuasi} is satisfied around the umbilics of the immersed sphere. Specifically, let $\Sigma$ be an immersed sphere in $\R^3$ such that, in a neighborhood of each umbilic $p\in \Sigma$, the inequality \eqref{cuasi} holds for $c:=\kappa_i(p)$, $i=1,2$, and for some $\Lambda=\Lambda(p)\leq -1$. Then $\Sigma$ is a round sphere. Here, no assumption is made on $\Sigma$ away from its umbilic set, and the values of the umbilicity constant $c$ and the quasiconformal coefficient $\Lambda$ are allowed to be different on different umbilics.
\end{remark}

Theorem \ref{th:main1} answers a question posed by H. Hopf in 1956 on the curvature diagram of immersed spheres in $\R^3$, that we explain next. Recall that the \emph{curvature diagram} of a surface $\Sigma$ is the region $\mathcal{D}:=\{(\kappa_1(p),\kappa_2(p)) : p \in \Sigma\}\subset \R^2$, where $\kappa_1\geq \kappa_2$ are the principal curvatures of $\Sigma$. Note that $\cD$ lies in the closed half-plane $x\geq y$ of $\R^2$, and the points where $\cD$ intersects the diagonal $y=x$ correspond to the umbilics of $\Sigma$. In particular, if $\Sigma$ is compact of genus zero, then $\cD$ intersects the diagonal (since $\Sigma$ must have some umbilic, by Poincaré-Hopf theorem).

In his famous book \cite{Ho}, Hopf proved that a compact, \emph{real analytic} surface of genus zero  $\Sigma$ in $\R^3$ must be a round sphere if its curvature diagram $\cD$ satisfies the following property ${\bf (H)}$: \emph{$\cD$ has cusps at the diagonal $y=x$, with tangents orthogonal to this diagonal}. That is, $\cD$ lies in a region of the half-plane $x\geq y$ that has such cusps at the diagonal (see Figure \ref{fig:Hopf}, left). He asked then if the analyticity condition can be removed, see \cite[p. 145]{Ho}. Our Theorem \ref{th:main1} together with Remark \ref{rem:1} gives a positive answer to Hopf's question:

\begin{corollary}[Solution to Hopf's problem]\label{cor:menosuno}
Any immersed sphere in $\R^3$ whose curvature diagram satisfies property ${\bf (H)}$ is a round sphere.
\end{corollary}

More specifically, our results imply a positive solution to Hopf's problem under a much weaker condition:

\begin{corollary}\label{cor:0}
An immersed sphere $\Sigma$ in $\R^3$ must be a round sphere if its curvature diagram $\cD$ satisfies the following property ${\bf (W)}$: \emph{$\cD$ has wedges of negative slopes at the diagonal (not necessarily cusps)}. See Figure \ref{fig:Hopf}, right.
\end{corollary}
\begin{proof}
The hypothesis that $\Sigma$ satisfies property ${\bf (W)}$ is easily seen to be equivalent to the curvature condition imposed in Remark \ref{rem:1}. Thus, Corollary \ref{cor:0}, and so, Corollary \ref{cor:menosuno}, hold.
\end{proof}

\begin{figure}[htbp]
    \centering
    \includegraphics[width=0.4\textwidth]{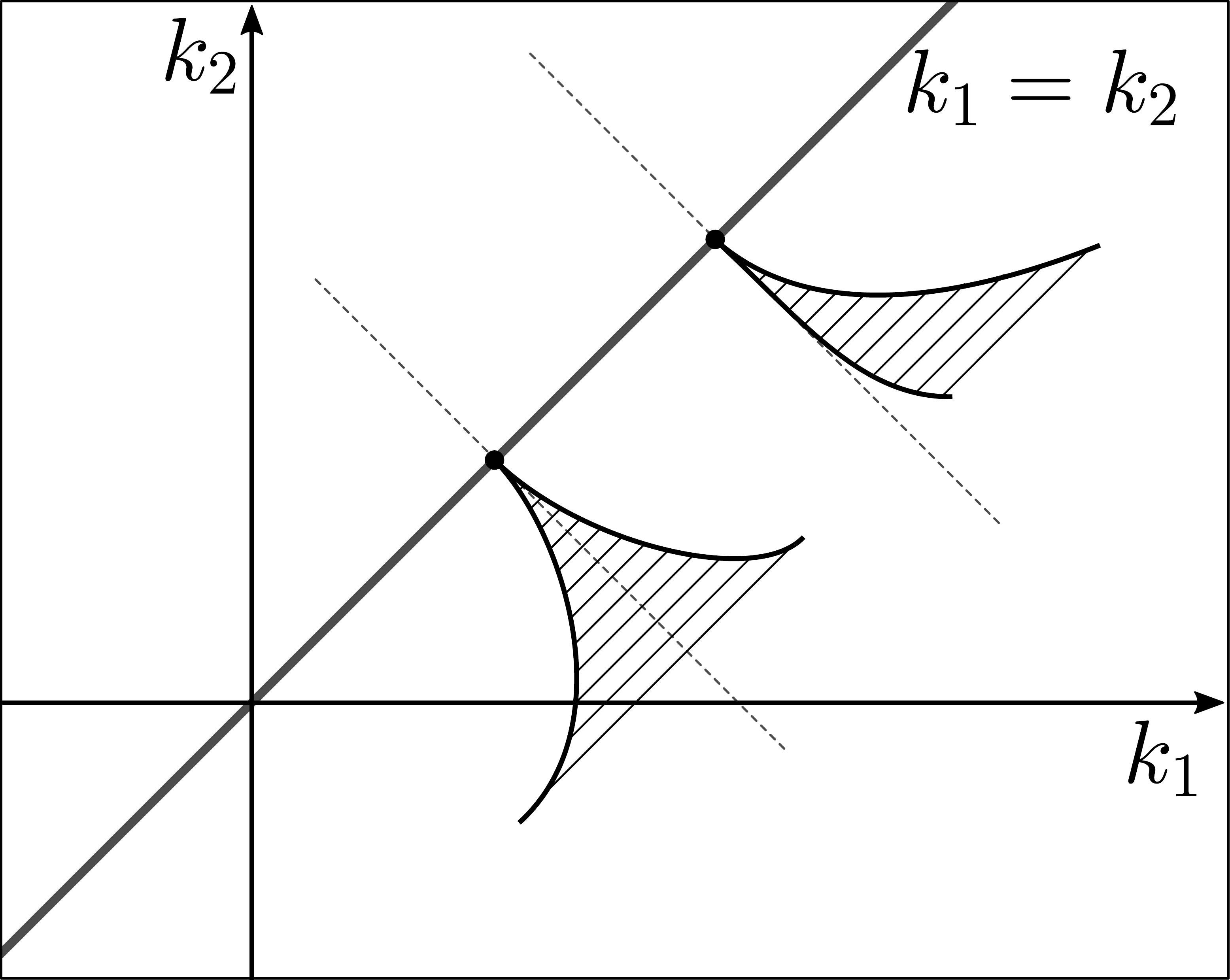} \hspace{0.5cm}  \includegraphics[width=0.4\textwidth]{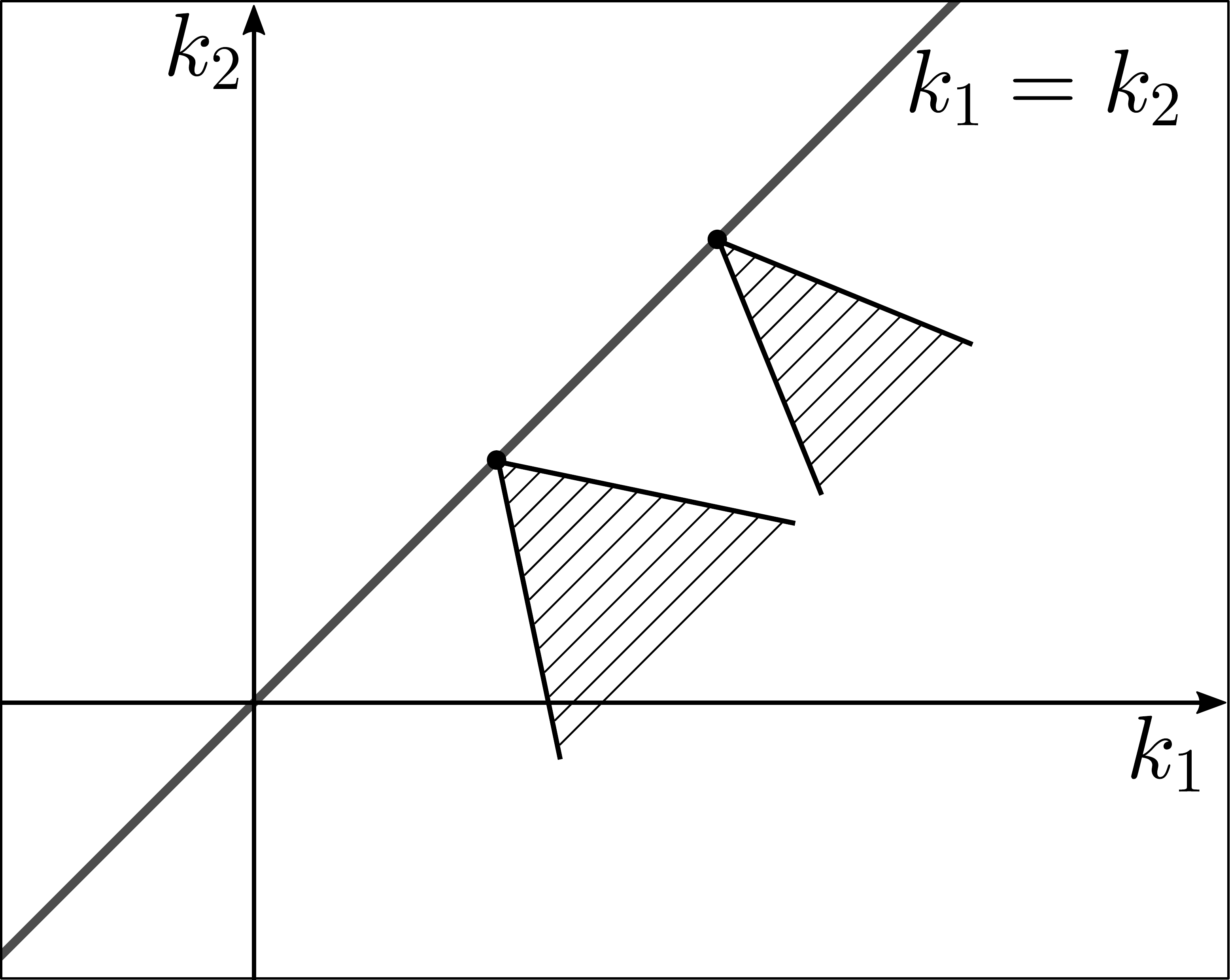}
    \caption{Left: Curvature diagram restriction ${\bf (H)}$ in Hopf's problem. Right: Curvature diagram restriction ${\bf (W)}$, with wedges of negative slopes at the diagonal. Note that ${\bf (W)}$ includes ${\bf (H)}$ as a particular case.}
    \label{fig:Hopf}
\end{figure}

The proof of Theorem \ref{th:main1} relies on the Poincaré-Hopf theorem. Specifically, we will show (Section \ref{sec:main}) that the umbilics of a quasi-CMC surface in $\R^3$ are isolated, and have non-positive index, unless the surface is totally umbilical. This proves the theorem. However, the proof of these properties is not elementary, due to the generality of the quasi-CMC condition \eqref{cua0}. In order to show that they hold, we will first develop in Section \ref{sec:an} an analytic study of the Hessian $D^2 u$ around critical points of smooth solutions $u$ to linear, uniformly elliptic equations $L[u]=0$ with bounded (non-continuous) coefficients. More specifically, we will show that ${\rm det}(D^2 u)<0$ in a punctured neighborhood of any such critical point, and that the vector fields $\nabla u_x$ and $\nabla u_y$ have non-positive index. The key tool for proving this result will be the Bers-Nirenberg representation, \cite{BN}. The analytic results of Section \ref{sec:an} might have further geometric applications.

In Section \ref{sec:ani} we will give an \emph{anisotropic} version of Theorem \ref{th:main1}, see Theorem \ref{th:aniso}, showing that any immersed sphere $\Sigma$ in $\R^3$ that is quasiconformal with respect to a given ovaloid $S_0$ must be a translation of this ovaloid. This extends to the genus zero case an important result by Han, Nadirashvili and Yuan \cite{HNY}, who proved (via the study of $1$-homogeneous solutions to a uniformly elliptic equation in $\R^3$) that the statement of Theorem \ref{th:aniso} holds when \emph{both} surfaces $\Sigma, S_0$ are ovaloids; see Section \ref{sec:ani} for more details on this connection.

Theorems \ref{th:main1} and \ref{th:aniso} seem sharp in a number of directions, and have additional geometric consequences. These will be discussed in Section \ref{sec:2} below, and in an Appendix.

\section{Discussion, sharpness and consequences of the result}\label{sec:2}

Theorem \ref{th:main1} has two topological hypotheses that are unavoidable for its validity. On the one hand, the result is strictly two-dimensional, since there exist immersed, non-round CMC hypersurfaces in $\R^{n+1}$ that are diffeomorphic to $\S^n$ for any $n\geq 3$, see \cite{Hs}. On the other hand, the genus zero hypothesis cannot be removed, since there exist immersed compact CMC surfaces of arbitrary genus in $\R^3$. Moreover, as already pointed out in the introduction, a thin tube over a simple, closed regular curve in $\R^3$ is an example of a closed, embedded quasi-CMC surface of genus one in $\R^3$. So, the result is topologically sharp.

The quasiconformal condition \eqref{cuasi} is related to a classical conjecture by A.D. Alexandrov \cite{A0,A}, which can be formulated in the following way: \emph{if a closed, $C^2$ convex surface $S\subset \R^3$ has its principal curvatures $\kappa_1,\kappa_2>0$ satisfying at any point 
 \begin{equation}\label{ecur1}
 (\kappa_1-c)(\kappa_2-c)\leq 0, \hspace{0.5cm} \text{and equality holds only if $\kappa_1=\kappa_2=c$},
 \end{equation}
for some constant $c>0$, then $S$ is a sphere of radius $1/c$.} Alexandrov proved this result in \cite{A0} for the case that $S$ is real analytic. The conjecture remained open for a long time, until Martinez-Maure \cite{MM} constructed in 2001 a $C^2$ counterexample to it; see also Panina \cite{Pa1}.

One should observe at this point that the quasi-CMC inequality \eqref{cuasi} implies the Alexandrov condition \eqref{ecur1}. More specifically, \eqref{cuasi} is actually the uniformly elliptic version of the degenerate elliptic condition \eqref{ecur1}. In this sense, it is interesting to remark that Theorem \ref{th:main1} does not hold under the weaker Alexandrov hypothesis \eqref{ecur1}. Indeed, in Lemma \ref{lem:ej2} of the Appendix we will construct many smooth non-round spheres embedded in $\R^3$ that satisfy the Alexandrov inequality \eqref{ecur1}. This indicates that the quasi-CMC condition \eqref{cuasi} in Theorem \ref{th:main1} is close to being sharp.

In the particular case where the quasi-CMC sphere $\Sigma$ is strictly convex, Theorem \ref{th:main1} can be obtained as a direct consequence of the uniqueness theorem by Han, Nadirashvili and Yuan \cite{HNY} mentioned above; see also \cite[Sect. 1.6]{NTV}. We note that the approach in \cite{HNY} needs the convexity assumption, and so it cannot handle the general case of immersed spheres that we study here.

In the real analytic case, Theorem \ref{th:main1} was previously known, after a theorem of Voss and Münzner (see Satz III in \cite{Mu}): \emph{Any real analytic immersed sphere in $\R^3$ satisfying $(\kappa_1-c)(\kappa_2-c)\leq 0$ for some $c>0$ is a sphere of radius $1/c$.} The analyticity assumption in this result cannot be removed, by the examples given in Lemma \ref{lem:ej2}, or in Panina \cite{Pa1}.

An immediate, non-trivial consequence of Theorem \ref{th:main1} is the characterization of round spheres as the only compact elliptic Weingarten surfaces of genus zero, under much weaker hypotheses on the Weingarten equation than the previous classification results for this type of surfaces \cite{Ho,Ch,HW1,Br,GM3}. Specifically, we obtain the result below directly from Corollary \ref{cor:0}:

\begin{corollary}\label{cor1}
Let $\Sigma$ be an immersed sphere in $\R^3$ whose principal curvatures $\kappa_1\geq \kappa_2$ satisfy a Weingarten equation $\kappa_1= f(\kappa_2)$, for some real function $f$. Assume that $f$ is continuous and its Dini derivatives are negative and finite at every point. Then, $\Sigma$ is a round sphere.
\end{corollary}
\begin{proof}
The Dini condition on $f$ automatically implies that the curvature diagram of $\Sigma$, which is a subset of the curve $\kappa_1=f(\kappa_2)$, approaches the principal diagonal of the $(\kappa_1,\kappa_2)$-plane inside a wedge of negative slopes, and so the result follows from Corollary \ref{cor:0}.
\end{proof}

Note that if $f\in C^1$, the Dini condition is equivalent to $f'<0$, and we recover the general classification of elliptic Weingarten spheres in \cite{GM3}. Also, note that the equation $\kappa_1=f(\kappa_2)$ is not symmetric in $(\kappa_1,\kappa_2)$; therefore, Corollary \ref{cor1} (or the classification in \cite{GM3}) is not covered by the classical works of Hopf, Chern, Hartman-Wintner and Bryant about the classification of elliptic Weingarten spheres. 

We single out another interesting consequence of Theorem \ref{th:main1}. It gives a non-trivial property of any immersed sphere in $\R^3$ with positive (not constant) mean curvature.

\begin{corollary}\label{corh}
If an immersed sphere $\Sigma$ in $\R^3$ has mean curvature $H\geq 1$ at every point, then unless $\Sigma$ is a sphere of radius $1$, there should exist some point of $\Sigma$ with $K>1$.
\end{corollary}
\begin{proof}
Let $\Sigma$ be an immersed sphere in $\R^3$ for which $K\leq 1\leq H$ holds. Then, $\Sigma$ satisfies \eqref{cua0} for $c=1$ and $$\mu= \frac{H_0-1}{H_0+1}, \hspace{0.5cm} H_0:= {\rm max}_{\Sigma} \, H.$$ Indeed, for that value of $\mu$ we have by $K\leq 1$ that $$(H-1)^2-\mu (H^2-K)\leq H^2 (1-\mu) - 2 H + 1 + \mu,$$ and the right hand-side is $\leq 0$ for all values $H\in [1,H_0]$. Therefore, $\Sigma$ is a sphere of radius $1$, by Theorem \ref{th:main1}.
\end{proof}

Observe again that Corollary \ref{corh} is not true without the genus zero assumption; take for example a thin rotational tube over a large circle. Note that Corollary \ref{corh} is not an elementary result, since it contains as a particular case Hopf's soap bubble theorem (due to the general inequality $H^2\geq K$).

Some of the arguments that we present in this paper carry over naturally to the uniqueness study of immersed spheres in Riemannian $3$-manifolds, by means of the notion of \emph{transitive family} of surfaces, see \cite{GM3,GM4}. Our index study here also gives relevant information about overdetermined elliptic problems, as in \cite{GM5,Mi,EM}. However, these lines of inquiry will not be pursued here.

\section{On the Hessian of solutions to linear elliptic equations}\label{sec:an}

Consider the linear homogeneous equation
\begin{equation}\label{linpde}
L[u]:=a_{11} u_{xx} + 2 a_{12} u_{xy} + a_{22} u_{yy} +b_1 u_x + b_2 u_y =0,
\end{equation}
where the coefficients are bounded, measurable functions in some domain $\Omega\subset \R^2$, that satisfy the uniform ellipticity condition 
 \begin{equation}\label{elipcon}
\landa_1 |\xi|^2 \leq \sum a_{ij} \xi_i \xi _j \leq \landa_2 |\xi|^2 \hspace{1cm} \forall \xi=(\xi_1,\xi_2)\in \R^2,
 \end{equation}
for some positive constants $0<\landa_1\leq \landa_2$. Note that no regularity is assumed for the coefficients of \eqref{linpde}. In particular, they might be discontinuous.

In \cite{BN}, Bers and Nirenberg showed that critical points of solutions $u$ to \eqref{linpde} are isolated, and the gradient $\nabla u$ has around such critical points a zero of finite order, and a nodal structure equivalent to that of a holomorphic function. In Theorem \ref{lemanal2} below we use the Bers-Nirenberg representation in \cite{BN} to study the Hessian $D^2u$ of smooth solutions to \eqref{linpde}. We remark that when the coefficients $a_{ij}$ of \eqref{linpde} are $C^1$, or more generally Holder continuous, Theorem \ref{lemanal2} is known and follows from results of Hartman-Wintner \cite{HW} and Bers \cite{Be}.

Also, note that a $C^{\8}$ regularity assumption on the solution $u$ is not unnatural, even if the coefficients of \eqref{linpde} are discontinuous. For instance, consider two linear, uniformly elliptic homogeneous operators $L_i[u]$, $i=1,2$, as in \eqref{linpde}, with smooth coefficients, and let $u(x,y)$ be a smooth function on $\Omega$ satisfying $$L_1[u]\leq0\leq L_2[u].$$ Then, there exists a uniformly elliptic operator $\cL[u]$ as in \eqref{linpde}, but this time with discontinuous coefficients in general, such that $\cL[u]=0$. See also Theorem \ref{anath} below.

\begin{theorem}\label{lemanal2}
Let $u\in C^{\8}(\Omega)$ be a non-constant solution to \eqref{linpde}, with $\nabla u(p_0)=(0,0)$ at some $p_0\in \Omega$. Then, there exists a punctured disk $\cD^*\subset \Omega$ centered at $p_0$ such that:
\begin{enumerate}
\item
${\rm det}(D^2 u)<0$ in $\cD^*$.
 \item
The (common) index of the gradient vector fields $\nabla u_x$, $\nabla u_y$ around $p_0$ is non-positive.
\end{enumerate}
\end{theorem}
\begin{proof}
We start with some well known manipulations. By \eqref{elipcon}, $a_{11}+a_{22}$ is bounded between two positive constants. Dividing \eqref{linpde} by this quantity, we can rewrite it in terms of the complex parameter $z=x+iy$ as
 \begin{equation}\label{pdez}
 2 u_{z\bar{z}} + \mu u_{zz} + \bar{\mu} u_{\bar{z}\bar{z}}+ \beta u_z + \overline{\beta} u_{\bar{z}}=0,
  \end{equation} where 
  \begin{equation}\label{ecmu}
  \mu := \frac{a_{11}-a_{22}+2i a_{12}}{a_{11}+a_{22}},\hspace{1cm} \beta := \frac{b_1+ib_2}{a_{11}+a_{22}}.
  \end{equation}
In this way, if $0<\mu_1(z)\leq \mu_2(z)$ denote the eigenvalues of $(a_{ij})$ at a point $z\in \Omega$, then we have at $z$ $$|\mu| = \frac{K_{\mu}-1}{K_{\mu}+1}, \hspace{1cm} K_{\mu}:= \frac{\mu_2}{\mu_1}.$$
In particular, it holds
 \begin{equation}\label{inemu}
 |\mu|\leq \mu_0:= \frac{\cK-1}{\cK+1} <1, \hspace{1cm} \cK := \frac{\landa_2}{\landa_1} \geq 1.
 \end{equation}

Let now $p_0\in \Omega$ be a critical point of $u$, and assume $p_0=(0,0)$ for simplicity. Consider a simply connected domain containing the origin and whose closure is contained in $\Omega$. For simplicity, as our study is local, we still denote this domain by $\Omega$. In these conditions we may apply the Bers-Nirenberg representation theorem in \cite{BN} (see also \cite{Be2,BJS}), which says that the complex gradient $f=u_z$ of any solution $u$ to \eqref{pdez} can be written in $\Omega$ as 
 \begin{equation}\label{bn}
 f(z)= e^{s(z)} F(\chi(z)),
 \end{equation}
where:
\begin{enumerate}
\item[(i)]
$s:\Omega\flecha \C$ is a Holder continuous function. It is equal to zero when $\beta=0$.
 \item[(ii)]
$\chi:\Omega\flecha \Omega':=\chi(\Omega)\subset \C$ is a $\cK$-quasiconformal homeomorphism, that satisfies the Beltrami equation $\chi_{\bar{z}}=\mu \chi_z$ for $\mu$ given by \eqref{ecmu}. In particular, both $\chi$ and $\chi^{-1}$ are of class $C^{\alfa}$, with $\alfa=1/\cK$; see, e.g., \cite{AIM}.
 \item[(iii)]
$F$ is a holomorphic function on $\Omega'$.
\end{enumerate}

We proceed with the rest of the proof. If $D^2 u$ is not zero at the origin, then item (1) follows by the ellipticity of \eqref{linpde}, and item (2) is trivial since $\nabla u_x, \nabla u_y$ do not vanish around the origin, and so their index is zero. So, assume from now on that $D^2u=0$ at the origin, i.e., $Df(0)=0$. Also, $f(0)=0$ since the origin is a critical point of $u$. In \eqref{bn}, denote $\xi_0:= \chi(0)$, and let $n\geq 1$ denote the order of the zero of $F(\zeta)$ at $\zeta_0$ (note that $F(\zeta_0)=0$ and that $F\not\equiv 0$ since $u$ is non-constant).

It follows then from \eqref{bn} and the fact that both $\chi$ and $\chi^{-1}$ are Holder with exponent $1/\cK$ that, in a sufficiently small neighborhood of the origin, we have 
 \begin{equation}\label{2holder}
 c_1 |z|^{n\cK}\leq |f(z)| \leq c_2 |z|^{n/\cK},
 \end{equation}
for positive constants $c_1,c_2$.

Thus, since $f\in C^{\8}(\Omega)$, we have by the left inequality in \eqref{2holder} that there exists a first homogeneous non-zero term of degree $\nu\geq 2$ in the Taylor series expansion of $f$ at $0$. Then, we can write around the origin 
 \begin{equation}\label{uom}
f(z) = \omega(z) + o(|z|^{\nu})
 \end{equation} where $\omega=\omega_1+i\omega_2$ is a complex-valued homogeneous polynomial in $\R^2$ of degree $\nu$. 
 
 It follows from \eqref{pdez} that
 $$|f_{\bar{z}}|\leq \mu_0 |f_z| + c |f|,$$
 for some positive constant $c>0$, where $\mu_0$ is given by \eqref{inemu}. Therefore, dividing this inequality by $|z|^{\nu-1}$ and taking limits as $|z| \to 0$ we obtain from \eqref{uom} that
  \begin{equation}\label{eqom}
|\omega_{\bar{z}}|\leq \mu_0 |\omega_z|.
\end{equation} 
Observe that $J(z,\omega):=|\omega_z|^2-|\omega_{\bar{z}}|^2 \geq 0$ by \eqref{eqom}, since $\mu_0<1$. We claim that $J(z,\omega)>0$ in $\R^2-\{(0,0)\}$. 

Indeed, assume that $J(z,\omega)=0$ at some $\eta_0\in \C-\{0\}$. 
By \eqref{eqom}, we have $D\omega(\eta_0)=0$. By Euler's formula for homogeneous functions, $\omega(\eta_0)=0$. Thus, $\omega$ admits a factorization of the form $\omega(z)=(\alfa_0 z + \overline{\alfa_0} \bar{z})^m Q(z)$, with $m\geq 2$, where $\alfa_0:=i \overline{\eta_0}$, and $Q$ is a complex-valued homogeneous polynomial with $Q(\eta_0)\neq 0$. Thus, away from the line $L \equiv \alfa_0 z + \overline{\alfa_0} \bar{z}=0$ in $\C$ we have

$$\frac{|\omega_{\bar{z}}|}{|\omega_z|} = \frac{|m \, \overline{\alfa_0}\,Q + (\alfa_0 z + \overline{\alfa_0} \bar{z}) Q_{\bar{z}}|}{| m \, \alfa_0\, Q + (\alfa_0 z + \overline{\alfa_0} \bar{z}) Q_{z}|}.$$ This quantity converges to $1$ as we approach $L$, a contradiction with \eqref{eqom}.

Therefore, $J(z,\omega)>0$ except at the origin. A standard blow-up argument from \eqref{uom} shows then that $J(z,f)>0$ on a sufficiently small punctured disk $D^*(0,\ep_0)$, i.e., that $u_{xx} u_{yy}-u_{xy}^2<0$ in that punctured disk. This proves item (1) of the statement. Observe that this implies that the origin is an isolated critical point of both $u_x$ and $u_y$.

To finish the proof, let us show the non-positivity of the topological index of the vector fields $\nabla u_x$ and $\nabla u_y$ around the origin. Here, we recall that if a $C^1$ function $w(x,y)$ has an isolated critical point at $p_0$, the topological index of its gradient $\nabla w$ at $p_0$ is given by the winding number of $\nabla w$ around a sufficiently small circle centered at $p_0$.

First of all, note that if $\nabla u_x$ and $\nabla u_y$ were collinear at some point in $D^*(0,\ep_0)$, we would have $u_{xx}u_{yy}-u_{xy}^2=0$ at that point, which does not happen. Thus, the vector fields $\nabla u_x$, $\nabla u_y$ have the same index around the origin. By the same argument, this index also agrees with that of $\nabla u_{\theta}$, where $u_{\theta}:= \cos \theta u_x + \sin \theta u_y$, $\theta \in [0,2\pi)$.

Assume that the index of $\nabla u_x$ at the origin is positive. By a known argument (see \cite[Lemma 3.1]{AM}), this index is equal to $1$, and either $u_x>u_x(0,0)$ or $u_x<u_x(0,0)$ in a small punctured neighborhood of the origin. So, since the origin is a critical point of $u$, either $u_x>0$ or $u_x<0$. Moreover, one of the analogous inequalities holds for $u_{\theta}$ and each $\theta \in [0,2\pi)$, by the previously observed invariance of the index of $\nabla u_{\theta}$. By continuity then, one of these two inequalities should hold around the origin for all values of $\theta$. For definiteness, assume that $u_{\theta}(x,y)>0$ in a punctured neighborhood of the origin, for all $\theta$. This implies that, in a sufficiently small disk $D(0,\delta)$, we have $$u(x,y)\geq u(0,0),$$ which gives a contradiction with the maximum principle for \eqref{linpde}. Therefore, the common index of $\nabla u_x$ and $\nabla u_y$ is non-positive, what completes the proof of Theorem \ref{lemanal2}.
\end{proof}

The above arguments actually yield a more general result, that we present below, which controls the behavior of solutions to a natural differential inequality around their zeros. The statement of Theorem \ref{anath} below in the particular case $\mu_0=0$ is well known, and a direct consequence of the Bers-Vekua similarity principle. In that situation, the solution $f$ is asymptotically holomorphic around its zeros. However, when $\mu_0\neq 0$, this holomorphic behavior does not hold in general, as the example $f(z)=z |z|^2$ shows. The similarity principle for $\mu_0=0$ has been an important tool in the study of harmonic maps between surfaces, see e.g. Jost \cite[p. 75]{J}. In this sense, although Theorem \ref{anath} will not be used in this paper, we believe that it might have further geometric applications. For related analytic results, see e.g. Chapter 7 of Schulz \cite{Sch}. 
\begin{theorem}\label{anath}
Let $f$ be a complex, non-zero smooth function in $\Omega\subset \C$ that satisfies in a neighborhood of $z_0\in \Omega$ the inequality
\begin{equation}\label{anaine}|f_{\bar{z}}|\leq \mu_0 |f_z| + c |f|,
\end{equation} 
where $\mu_0<1$ and $c>0$ are constants. Assume $f(z_0)=0$. Then, there exists a punctured disk $D^*\subset\Omega$ centered at $z_0$ so that:

\begin{enumerate}
\item
$J(z,f):=|f_z|^2-|f_{\bar z}|^2>0$ holds in $D^*$.
\item
${\rm Re} (f)$, ${\rm Im} (f)$ have at $z_0$ an isolated critical point of non-positive index.
\end{enumerate}
\end{theorem}

\begin{proof}
For each $z\in \Omega$, there exists $\tau=\tau(z)\in [0,1]$ such that $$|f_{\bar{z}}|= \tau( \mu_0 |f_z| + c |f|),$$ as a consequence of \eqref{anaine}. Observe that, at points where $f=Df=0$, the value of $\tau$ is not uniquely determined, and so it can be chosen arbitrarily. In particular $\tau$ is not continuous. Write next $$f= e^{i\theta_0}|f|,\hspace{0.5cm} f_z = e^{i\theta_1} |f_z|, \hspace{0.5cm} f_{\bar{z}}= e^{i\theta_2} |f_{\bar{z}}|,$$ where again the values $\theta_j=\theta_j(z)\in [0,2\pi)$, $j=0,1,2$, are chosen arbitrarily at the zeros of their corresponding functions. We have then
 \begin{equation}\label{ani2}
 f_{\bar{z}} = \mu f_z + \alfa f , \hspace{1cm} \mu:=\mu_0 \tau e^{i(\theta_2-\theta_1)}, \hspace{0.2cm} \alfa:=c\, \tau e^{i(\theta_2-\theta_0)} .
 \end{equation}
Clearly, $${\rm sup} \{ |\mu(z)|: z\in \Omega\} \leq \mu_0<1, \hspace{1cm} {\rm sup}\{|\alfa(z)|:z\in \Omega\}\leq c.$$In particular, \eqref{ani2} is a linear uniformly elliptic system with bounded coefficients in the conditions of the Bers-Nirenberg representation \cite{BN}. Therefore, $f$ admits a representation \eqref{bn} around $z_0$.

Now, regarding assertion (1) in the statement, note that \eqref{bn} implies that $f$ satisfies \eqref{2holder}.  By the same arguments in the proof of Theorem \ref{lemanal2}, we conclude that $J(z,f)>0$ in a punctured disk $D^*\subset \Omega$ around $z_0$.

To prove assertion (2), write $f_1={\rm Re}(f)$, $f_2={\rm Im}(f)$. From $J(z,f)>0$ we obtain that $\nabla f_1$, $\nabla f_2$ are always linearly independent in $D^*$. In particular, $z_0$ is, at worst, an isolated critical point of both $f_1,f_2$, and the indices of the gradients $\nabla f_1,\nabla f_2$ at $z_0$ coincide. 

Assume that this index is positive. By \cite[Lemma 3.1]{AM}, we have that each of $f_1$ and $f_2$ is either always positive or always negative in $D^*$, that is, $f(D^*)$ lies in an open quadrant of $\C$. But next, recall that by \eqref{bn} we have $f=e^s (F\circ \chi)$ in $D^*$, where $s$ is continuous, $F$ is holomorphic and $\chi$ is a quasiconformal homeomorphism, all in $D=D^*\cup \{z_0\}$. Also, $F\circ \chi$ is an open mapping in $D$ that sends $z_0$ to the origin, and in particular, the image of the restriction of $F\circ \chi$ to a small enough circle $|z-z_0|=r$ winds around the origin a non-zero number of times. On the other hand, making $D^*$ smaller if necessary, we can assume that $e^s(D)$ is a small open neighborhood of some non-zero $\zeta_0\in \C^*$, and in particular the argument function of $e^s$ has an arbitrary small image along $|z-z_0|=r$. Thus, since ${\rm arg} (f) = {\rm arg }(e^s) + {\rm arg (F\circ \chi)}$, the variation of the argument of $f$ around any such circle can be made larger than $\pi/2$. This contradicts that $f(D^*)$ lies in a quadrant. Therefore, the index is non-positive, and this completes the proof.
\end{proof}

\section{Proof of Theorem \ref{th:main1} and Remark \ref{rem:1}}\label{sec:main}

Theorem \ref{th:main1} and Remark \ref{rem:1} will be an almost direct consequence of a local result (Theorem \ref{localth}) that we prove below. To explain this, let us recall that, at any non-umbilical point of a surface $\Sigma$ in $\R^3$, there exist exactly two eigenlines $\cL_1,\cL_2$ of its second fundamental form, which describe the principal directions of $\Sigma$ at the point. Thus $\cL_i$, $i=1,2$, are continuous line fields on $\Sigma- \cU$, where $\cU$ is the umbilic set of $\Sigma$. In addition, if $p\in \cU$ is an \emph{isolated} umbilic, the principal line fields $\cL_1, \cL_2$ have a common well defined index at $p$, which is a half-integer. This number is called the \emph{index} of the isolated umbilic $p$. We have:

\begin{theorem}\label{localth}
Let $p$ be an umbilic of a surface $\Sigma$ in $\R^3$, let $c:=\kappa_i(p)$, $i=1,2$, and assume that \eqref{cuasi} holds in a neighborhood of $p$, for some $\Lambda<0$. Then, either $\Sigma$ is totally umbilical around $p$, or $p$ is an isolated umbilic of non-positive index.

In particular, any quasi-CMC surface in $\R^3$ is either totally umbilical, or it has only isolated umbilics of non-positive index.
\end{theorem}

Let us explain how Theorem \ref{th:main1} and Remark \ref{rem:1} follow from Theorem \ref{localth}. Assume that $\Sigma$ is an immersed sphere in $\R^3$ that satisfies \eqref{cuasi} around each umbilic $p\in \Sigma$, where $c=c(p)$ and $\Lambda=\Lambda(p)$ may have different values on different umbilics. 

Suppose first of all that there exists a connected component $\cV\subset \Sigma$ of the totally umbilical set of $\Sigma$ that has non-empty interior, and take $p\in \parc \cV$. Then, $p$ is also an umbilic of $\Sigma$, and by Theorem \ref{localth}, $\Sigma$ is totally umbilical around $p$. Thus, the closure of $\cV$ is an open set, and by connectedness, $\Sigma$ must be a totally umbilical round sphere.

Assume next that $\Sigma$ does not contain totally umbilical open sets. Then, by Theorem \ref{localth}, the umbilic set $\cU$ of $\Sigma$ is finite, and the index of any $p\in \cU$ is non-positive. In this way, $\cL_1, \cL_2$ define two continuous line fields in $\Sigma$ with a finite number of singularities, all of them of non-positive index. By the Poincare-Hopf theorem, the sum of such indices equals the Euler characteristic of $\Sigma$. This gives a contradiction, since $\Sigma$ has positive Euler characteristic.

Thus, we only need to prove Theorem \ref{localth}, which we do next:
\begin{proof}[Proof of Theorem \ref{localth}.] Start by noting that \eqref{cuasi} can be rewritten as 
\begin{equation}\label{cuasim}
m_2(\kappa_1-c)\leq \kappa_2-c \leq m_1 (\kappa_1-c), \hspace{0.5cm} \{m_1,m_2\} := \{\Lambda \pm \sqrt{\Lambda^2-1}\} <0.
\end{equation}
Consider for $i=1,2$ the Weingarten functionals 
 \begin{equation}\label{def:wei}
\cW_i (\kappa_1,\kappa_2):= -m_i(\kappa_1-c) + \kappa_2 -c, 
 \end{equation} 
defined on the set $\{\kappa_1\geq \kappa_2\}\subset \R^2$. 
Note that, for $i=1,2$,
 \begin{equation}\label{eliuni}
\frac{\parc \cW_{i}}{\parc \kappa_1} = -m_i >0, \hspace{0.5cm} \frac{\parc \cW_{i}}{\parc \kappa_2} =1.
  \end{equation}

\noindent This implies that the fully nonlinear operators $\cW_{i}$ are uniformly elliptic. 

In what follows, let $p\in \Sigma$ denote an umbilic, with principal curvatures equal to $c$. By hypothesis, \eqref{cuasim} holds in a neighborhood $V\subset \Sigma$ of $p$, and so, for each $q\in V$ there exists a value $\tau(q)\in [0,1]$ such that 
 \begin{equation}\label{weit}
 \cW^{\tau} (\kappa_1(q),\kappa_2 (q),q) =0, 
 \end{equation}
where 
 \begin{equation}\label{def:weit}
 \cW^{\tau} (\kappa_1,\kappa_2,q):= \tau (q) \cW_1(\kappa_1,\kappa_2) + (1-\tau(q)) \cW_2 (\kappa_1,\kappa_2).
 \end{equation}

\noindent One should observe here that $\tau$ is not continuous, and its value is undetermined at the points $q\in V$ where $\kappa_1=\kappa_2$, i.e., it can be chosen arbitrarily at those points. Note that, by \eqref{eliuni},
 \begin{equation}\label{eliweit}
0<-m_1\leq \frac{\parc \cW^{\tau}}{\parc \kappa_1} \leq -m_2, \hspace{0.5cm} \frac{\parc \cW^{\tau}}{\parc \kappa_2} =1,
  \end{equation}
which is a uniform ellipticity condition on $\cW^{\tau}$.
We can rewrite \eqref{def:weit} as 
\begin{equation}\label{def:weit2}
\cW^{\tau} = -c(1-m_2+\tau (m_2-m_1)) - \tau m_1 \kappa_1 -(1-\tau) m_2 \kappa_1 + \kappa_2
\end{equation}

Let $S_c$ denote the totally umbilical surface with principal curvatures equal to $c$, and that is tangent to $\Sigma$ at the umbilic $p$, with the same unit normal. \emph{We suppose from now on that $\Sigma$ and $S_c$ do not agree in a neighborhood of $p$}. Up to ambient isometry, for simplicity, we will assume that $p$ is the origin, and the unit normal of $\Sigma$ at $p$ is $e_3=(0,0,1)$. Let $z=u(x,y)$ and $z=u^0(x,y)$ be graphical local representations around the origin of $\Sigma$ and $S_c$, respectively. Thus,
$$u^0(x,y)=\frac{c(x^2+y^2)}{1+ \sqrt{1-c^2 (x^2+y^2)}}.$$

In this way, \eqref{weit} shows that $u(x,y)$ is a solution to a second order PDE
\begin{equation}\label{def:phitau}
\Phi^{\tau}[u]:=\Phi^{\tau} (x,y,u_x,u_y,u_{xx},u_{xy},u_{yy})=0.
\end{equation}
Here, $\Phi^{\tau}(x,y,p,q,r,s,t)$ is defined in a neighborhood of $(0,0,0,0,c,0,c)\in \R^7$. By \eqref{def:weit2}, its explicit expression is given by $$\Phi^{\tau}=  -c(1-m_2+\tau (m_2-m_1)) - \tau m_1 \cK_1 -(1-\tau) m_2 \cK_1 + \cK_2,$$
where $\tau=\tau(x,y)$, $0\leq \tau\leq 1$, $\{\cK_1,\cK_2\}:=\{\cH\pm \sqrt{\cH^2-\cK}\}$, and
\begin{equation}\label{eq:hk}
\cH(p,q,r,s,t):= \frac{r(1+q^2)-2pq s + (1+p^2)t}{2(1+p^2+q^2)^{3/2}}, \hspace{0.5cm} \cK(p,q,r,s,t):=\frac{r t-s^2}{(1+p^2+q^2)^2}.
\end{equation}

It is classical that the functions $\cK_i(p,q,r,s,t)$, $i=1,2$, have bounded first derivatives in a neighborhood of any point of the form $(0,0,c,0,c)\in \R^5$; see e.g. \cite[Section 5]{FGM} for the explicit computation. So, $\Phi^{\tau}(x,y,p,q,r,s,t)$ has the same property. In addition, it is immediate that $u^0(x,y)$ is also a solution to \eqref{def:phitau}, as $S_c$ satisfies $\cW_i(\kappa_1,\kappa_2)=0$ for $i=1,2$. 

So, if we denote $h=u-u^0$, a standard application of the mean value theorem to $\Phi^{\tau}$ shows that the restriction of $h$ to some convex neighborhood of the origin satisfies a linear, homogeneous equation with bounded coefficients 
$$L[h]:=a_{11}h_{xx} + 2 a_{12}h_{xy} + a_{22}h_{yy} + b_1h_x + b_2h_y =0.$$ Here, for each $(x,y)$ fixed in such a neighborhood, $$a_{11} (x,y):= \int_0^1 \frac{\parc \Phi^{\tau}}{\parc r} (x,y, u_p^{\vartheta},u_q^{\vartheta},u_r^{\vartheta},u_s^{\vartheta},u_t^{\vartheta} )d\vartheta,$$ where $u^{\vartheta}:= \vartheta u(x,y) + (1-\vartheta) u_0(x,y)$, etc., with similar formulas for the rest of the coefficients. Note that these coefficients are not continuous in $(x,y)$, in general.

Moreover, the ellipticity condition \eqref{eliweit} on $\cW^{\tau}$ actually implies that the operator $L[h]$ is uniformly elliptic, i.e., it satisfies \eqref{elipcon} around the origin, for adequate constants $0<\landa_1\leq \landa_2$. This is a standard fact, see e.g. Alexandrov \cite{A2}, which can also be checked by direct computation from \eqref{eliweit} and the definition of the coefficients $a_{ij}$.

Also, note that $\nabla h=(0,0)$ at the origin. Therefore, $h(x,y)$ is in the conditions of Theorem \ref{lemanal2}. In particular, since $h$ is not constant (recall that we are assuming that $\Sigma$ is not totally umbilical around $p$), $h$ has a zero of finite order $k>2$ at $(0,0)$ (by Bers-Nirenberg) and we have from Theorem \ref{lemanal2} that $h_{xx} h_{yy}-h_{xy}^2<0$ in a punctured disk $D^*\subset \Omega$ centered at the origin. Observe that $k>2$ since the origin is an umbilic of principal curvatures equal to $c$ of both $z=u(x,y)$ and $z=u^0(x,y)$, and so $D^2 h$ vanishes at $(0,0)$.

Define next the analytic functions $$\Psi_1 (p,q):=\frac{p}{\sqrt{1+p^2+q^2}},\hspace{0.5cm} \Psi_2 (p,q):=\frac{q}{\sqrt{1+p^2+q^2}}.$$ Then, we can write for any pair $(p,q)$,  $(p_0,q_0)\in \R^2$ near the origin, and for $i=1,2$,
 \begin{equation}\label{menvt}
 \Psi_i(p,q)-\Psi_i (p_0,q_0)= \cA_{i1}(p-p_0)+\cA_{i2}(q-q_0),
 \end{equation}
where each $\cA_{ij}$ depends analytically on $(p,q,p_0,q_0)$, and $(\cA_{ij})= {\rm Id}$ at $(0,0,0,0)$. This follows, for instance, from the mean value theorem, or alternatively from a series expansion of the left hand-side of \eqref{menvt} with respect to $(p,q,p_0,q_0)$ around $(0,0,0,0)$.

Let us denote $(\alfa_{ij}):=II \cdot I^{-1}$, with $I,II$ being the first and second fundamental forms of $z=u(x,y)$, written with respect to the coordinates $(x,y)$. By a standard computation using the above notations and $\nabla u:=(u_x,u_y)$, we have
\begin{equation}\label{eq:a1}
(\alfa_{ij}) = \left(\def\arraystretch{1.3}\begin{array}{cc} \left(\Psi_1(\nabla u)\right)_x &  \left(\Psi_2(\nabla u) \right)_x \\  \left(\Psi_1(\nabla u)\right)_y&  \left(\Psi_2(\nabla u) \right)_y \end{array}\right).
\end{equation}

The same computation, but this time for $u^0(x,y)$, gives that 
\begin{equation}\label{eq:a2}
\left(\def\arraystretch{1.3}\begin{array}{cc} c & 0 \\ 0 & c \end{array}\right) = \left(\def\arraystretch{1.3}\begin{array}{cc} \left(\Psi_1(\nabla u^0)\right)_x &  \left(\Psi_2(\nabla u^0) \right)_x \\  \left(\Psi_1(\nabla u^0)\right)_y&  \left(\Psi_2(\nabla u^0) \right)_y \end{array}\right), 
\end{equation}
where we have used that $z=u^0(x,y)$ is totally umbilical, with principal curvatures equal to $c$.

Then, from \eqref{eq:a1} and \eqref{eq:a2}, using \eqref{menvt} and $h=u-u^0$, we obtain
$$\def\arraystretch{3} \begin{array}{lll}
(\alfa_{ij} -c \, \delta_{ij}) & = &\left(\def\arraystretch{1.3}\begin{array}{cc} \left(\Psi_1(\nabla u)-\Psi_1(\nabla u^0)\right)_x & \left(\Psi_2(\nabla u)-\Psi_2(\nabla u^0)\right)_x\\ \left(\Psi_1(\nabla u)-\Psi_1(\nabla u^0)\right)_y &\left(\Psi_2(\nabla u)-\Psi_2(\nabla u^0)\right)_y\end{array}\right) \\
 & = & \left(\def\arraystretch{1.3}\begin{array}{cc} \left(\cA^0_{11}\, h_{x} + \cA^0_{12}\, h_y\right)_x & \left(\cA^0_{21}\, h_{x} + \cA^0_{22} \,h_y\right)_x\\ \left(\cA^0_{11}\,h_{x} + \cA^0_{12} \, h_y\right)_y &\left(\cA^0_{21}\, h_{x} + \cA^0_{22} \, h_y\right)_y\end{array}\right),
 \end{array}$$
where $\cA^0_{ij}:= \cA_{ij}(u_x,u_y,u_x^0,u_y^0)$. Therefore, recalling that $(\cA_{ij})={\rm Id}$ when $(p,q,p_0,q_0)=(0,0,0,0)$, and that $h(x,y)$ has a zero of order $k>2$ at the origin, we obtain from the above expression that
\begin{equation}\label{eq:a3}
(\alfa_{ij} -c \, \delta_{ij})= \left(\def\arraystretch{1.3}\begin{array}{cc} h_{xx} & h_{xy}\\ h_{xy} &h_{yy} \end{array}\right) + o(\varrho^{k-2}),
\end{equation}
where $\varrho:=\sqrt{x^2+y^2}$. Hence, from \eqref{eq:a3}
 \begin{equation}\label{hhk}
 H^2-K=\frac{1}{4}(h_{xx}-h_{yy})^2 + h_{xy}^2 + o(\varrho^{2k-4})
 \end{equation}
where $H,K$ denote the mean and Gauss curvature of $z=u(x,y)$. Since $h_{xx}h_{yy}-h_{xy}^2<0$ in $D^*$, we have that $(h_{xx}-h_{yy})^2 + 4 h_{xy}^2$ is positive in $D^*$, and has a zero of order $2k-4$ at the origin. Thus, by \eqref{hhk}, $H^2-K>0$ holds in a maybe smaller punctured disk $D_0^*\subset D^*$ centered at the origin. This implies that $p$ is an isolated umbilic of $\Sigma$, as claimed in Theorem \ref{localth}.

To end the proof of Theorem \ref{localth}, we need to control the index of the umbilic $p$. In coordinates $(x,y)$, the principal line fields of $\Sigma$ around $p$ are given by the solutions to 
 \begin{equation}\label{plf}
 -\alfa_{12} dx^2 + (\alfa_{11}-\alfa_{22}) dx dy + \alfa_{21} dy^2 =0.
 \end{equation}
 Note that this equation remains invariant if we change $(\alfa_{ij})$ by $(\alfa_{ij}-c \, \delta_{ij})$. Thus, by \eqref{eq:a3}, it is clear that these line fields have around $p$ the same index as the line fields around $(0,0)$ given by  
 \begin{equation}\label{plf3}
 -h_{xy} (dx ^2-dy^2)+ (h_{xx}-h_{yy}) dx dy =0.
  \end{equation}
Note that the line fields \eqref{plf3} are well defined around the origin, since $h_{xx}h_{yy}-h_{xy}^2<0$ in $D^*(0,\rho_0)$. By a classical argument, the index of the line fields given by \eqref{plf3} is $\leq 0$ if and only if the index of the vector field $Z=(-2h_{xy},h_{xx}-h_{yy})$ is $\leq 0$; see Hopf \cite[p. 167]{Ho}. But now, note that $\esiz Z,\nabla h_y\esde <0$ in $D^*(0,\rho_0)$. In particular, $Z$ and $\nabla h_y$ have the same index at the origin. Thus, by Theorem \ref{lemanal2}, we deduce that the index of the principal line fields of $\Sigma$ around $p$ is $\leq 0$. 

This concludes the proof of Theorem \ref{localth}. As explained at the beginning of this section, this shows that Theorem \ref{th:main1} and Remark \ref{rem:1} hold.
\end{proof}

Let us remark that the non-positivity of the index of umbilics of quasi-CMC surfaces given in Theorem \ref{localth} cannot be improved to negativity, in contrast with the CMC case. See Lemma \ref{lem:eje1} in the Appendix.

\section{A uniqueness theorem for ovaloids and applications}\label{sec:ani}
We will next modify the arguments in the proof of Theorem \ref{th:main1} to obtain a more general result, in which the uniqueness property is not obtained for round spheres, but for an arbitrary ovaloid of $\R^3$.

Let $S_0$ be an ovaloid, i.e., a compact $C^{\8}$ surface in $\R^3$ with $K>0$ at every point. We say that an immersed surface $\Sigma$ in $\R^3$ is \emph{quasiconformal with respect to $S_0$} if the following inequality holds at every $q\in \Sigma$, for some $\Lambda\leq -1$:
 \begin{equation}\label{cuasis}
(\kappa_1-\kappa_1^0)^2 +(\kappa_2-\kappa_2^0)^2 \leq 2 \Lambda(\kappa_1-\kappa_1^0)\, (\kappa_2-\kappa_2^0).
\end{equation}
Here, $\kappa_1\geq \kappa_2$ are the principal curvatures of $\Sigma$ at $q$, and $\kappa_1^0 \geq \kappa_2^0$ are the principal curvatures of $S_0$ at the unique point $q_0\in S_0$ whose unit normal agrees with the unit normal $N(q)$ of $\Sigma$ at $q$. Note that, by their own definition, $\kappa_1^0,\kappa_2^0$ are viewed here as maps from $\Sigma$ to $\R$.

One should observe that if $S_0$ is a sphere of radius $1/c$, then \eqref{cuasis} coincides with \eqref{cuasi}, i.e., with the notion of quasi-CMC surface.

We have then the following \emph{anisotropic} extension of Theorem \ref{th:main1}.

\begin{theorem}\label{th:aniso}
Let $S_0$ be an ovaloid. Then, any immersed sphere $\Sigma$ in $\R^3$ that is quasiconformal with respect to $S_0$ is a translation of $S_0$.
\end{theorem}

For the particular case where the immersed sphere $\Sigma$ is also an ovaloid, Theorem \ref{th:aniso} follows from the uniqueness result of Han, Nadirashvili and Yuan in \cite{HNY}, as we explain next. Assume that $\Sigma$ is an ovaloid, and let $h,h_0$ be the support functions of $\Sigma$ and $S_0$, viewed as functions on $\S^2$. Then $h-h_0$ extends by homogeneity to a homogeneous function $v$ of degree $1$ in $\R^3-\{0\}$. It can be shown that $v$ satisfies a linear, uniformly elliptic equation $\sum A_{ij}D_{ij}v=0$ in $\R^3$ (with bounded, measurable coefficients) if and only if $\Sigma$ satisfies \eqref{cuasis}. By \cite{HNY}, any such $v$ is linear. So, $\Sigma$ is a translation of $S_0$. One should note that, even though the homogeneous equation in $\R^3$ for $v$ considered in \cite{HNY} is equivalent to \eqref{cuasis} (for ovaloids), the regularity on the solution $v$ imposed in \cite{HNY} is much weaker than the $C^{\8}$ regularity that we ask here. On the other hand, the result in \cite{HNY} does not cover the general case of immersed spheres given by Theorem \ref{th:aniso}.

\begin{proof}
Let $II$ be the second fundamental form of the \emph{fixed} surface $\Sigma$, and $II_0$ be the second fundamental form of the \emph{osculating} ovaloid $S_0$ to $\Sigma$ at each point.  Specifically, $II_0(q)$ is, for each $q\in \Sigma$, the second fundamental form of $S_0$ at the point $q_0\in S_0$ that has unit normal equal to $N(q)$. Note that we can view $II_0$ as a Riemannian metric on $\Sigma$. 

Let $\cU_0:=\{q \in \Sigma : II(q)=II_0(q)\}.$ That is,  $\cU_0$ is the set of points where the immersed sphere $\Sigma$ has a contact of order at least two with some translation of the ovaloid $S_0$.  

From now on, we fix an arbitrary point $p\in \cU_0\subset \Sigma$, and consider a sufficiently small neighborhood $V\subset \Sigma$ of $p$. We will begin following closely the proof of Theorem \ref{localth}.

To start, we rewrite \eqref{cuasis} as 
\begin{equation}\label{cuasiman}
m_2(\kappa_1-\kappa_1^0)\leq \kappa_2-\kappa_2^0 \leq m_1 (\kappa_1-\kappa_1^0),\end{equation}
for constants $m_i<0$, where $\kappa_1,\kappa_2,\kappa_1^0,\kappa_2^0$ are defined on $V$. Denote next, for $i=1,2$,
 \begin{equation}\label{def:weian}
\cW_i (\kappa_1,\kappa_2,q):= \kappa_2-\kappa_2^0 - m_i(\kappa_1-\kappa_1^0),
 \end{equation} 
both of them defined on the set $\{(x,y,q):x\geq y, q\in V\}\subset \R^2\times V$. 

By \eqref{cuasiman}, we have on $V$ that $\cW_1 (\kappa_1(q),\kappa_2(q),q)\leq 0$ and $\cW_2 (\kappa_1(q),\kappa_2(q),q)\geq 0$. So, for each $q\in V$ there exists a value $\tau=\tau(q)\in [0,1]$ such that 
 \begin{equation}\label{weitan}
 \cW^{\tau} (\kappa_1(q),\kappa_2 (q),q) =0 \hspace{0.5cm} \text{for all $q\in V$},
 \end{equation}
where 
 \begin{equation}\label{def:weitan}
 \cW^{\tau} (\kappa_1,\kappa_2,q):= \tau  \cW_1+ (1-\tau) \cW_2.
 \end{equation}

Choose next Euclidean coordinates $(x,y,z)$ in $\R^3$ so that both $\Sigma$ and $S_0$ can be seen, respectively, as graphs $z=u(x,y)$ and $z=u_0(x,y)$ in these coordinates. We can also assume that, in these coordinates, $p=(0,0,0)$, the common unit normal of $\Sigma$ and $S_0$ at $p$ is $(0,0,1)$, and $II(p) =II_0(p)$ is a (positive definite) diagonal matrix $B$. Then, by \eqref{weitan}, $u(x,y)$ is a solution to a second order PDE
\begin{equation}\label{def:phitauan}
\Phi^{\tau}[u]:=\Phi^{\tau} (x,y,u_x,u_y,u_{xx},u_{xy},u_{yy})=0.
\end{equation}

At this point we can discuss the ellipticity of \eqref{def:phitauan} in the very same way that we did in the proof of Theorem \ref{localth} for the equation \eqref{def:phitau}. Indeed, the definitions of $\cW_i$ and $\cW^{\tau}$ in this general case do not alter the dependence with respect to second order derivatives of the corresponding objects \eqref{def:wei}, \eqref{def:weit}, and so the uniform ellipticity discussion is the same as in Theorem \ref{localth}.

In particular, if we define $h:=u-u_0$, then arguing still as in Theorem \ref{localth}, there exists a small disk $D_{\ep}(0)$ around the origin so that $h(x,y)$ is a solution to a linear uniformly elliptic equation with bounded measurable coefficients 
$$L[h]:=a_{11}h_{xx} + 2 a_{12}h_{xy} + a_{22}h_{yy} + b_1h_x + b_2h_y =0.$$

Note that $h$, $\nabla h$ and $D^2 h$ vanish at $(0,0)$. \emph{Assume for the moment that $h$ is not identically zero around the origin}. Then, by Theorem \ref{lemanal2}, we have that $h_{xx}h_{yy}-h_{xy}^2<0$ in a punctured neighborhood of the origin, and that the vector fields $\nabla h_x$, $\nabla h_y$ have non-positive index at $(0,0)$. Moreover, $h(x,y)$ has at the origin a zero of finite order (by Bers-Nirenberg), and so we can write for $\varrho:=\sqrt{x^2+y^2}$
 \begin{equation}\label{finor}
 h(x,y)= w(x,y) + o(\varrho^k)
 \end{equation}
where $w(x,y)$ is a homogeneous polynomial of degree $k\geq 3$. It also holds for $w$ that ${\rm det}(D^2 w)<0$ in $\R^2-\{(0,0)\}$, and that the index of the vector fields $\nabla w_x$, $\nabla w_y$ is non-positive at the origin.

Following a previous construction by the first two authors in arbitrary Riemannian $3$-manifolds, cf. \cite[Eq. (3.3)]{GM3}, we consider the smooth symmetric bilinear form $$\sigma:=II-II_0 : T\Sigma\times T\Sigma\flecha C^{\8}(\Sigma),$$ which compares the \emph{fixed} second fundamental form $II$ with the \emph{osculating} second fundamental form $II_0$. In these conditions, if we write the expression of $\sigma$ with respect to the graphical coordinates $(x,y)$, equation (3.13) in \cite{GM3} shows that we have the asymptotic expansion
\begin{equation}\label{finorr}
\sigma=D^2 w(x,y) + o(\varrho^{k-2}).
\end{equation}
Here, the key point for the validity of \eqref{finorr} in our situation is that we already proved that $h(x,y)$ satisfies \eqref{finor} in our general quasiconformal setting; so we do not need that $\Sigma$ and $S_0$ satisfy a $C^{1,\alfa}$ elliptic PDE, as is the case in \cite{GM3}.

Consider finally the endomorphism $\alfa:\X(\Sigma)\flecha \X(\Sigma)$ given by $II_0(\alfa(X),Y)=\sigma(X,Y)$ for all $X,Y\in \X(\Sigma)$. Note that $\alfa$ is diagonalizable at every point, and the matrix expressions of $\alfa$, $\sigma$ and $II_0$ with respect to the $(x,y)$ coordinates are related by $\alfa=\sigma\cdot II_0^{-1}$. Also, denote in these coordinates $$B^{-1}= II_0^{-1}(p)=\left(\begin{array}{cc} \beta_1 & 0 \\ 0 & \beta_2\end{array}\right),$$ where $\beta_1,\beta_2>0$. Up to a rotation in the $(x,y)$-coordinates, we can assume that $\beta_1\leq \beta_2$. The line fields on $\Sigma$ described by the eigendirections of $\alfa$ are given by \eqref{plf}. Then, by \eqref{finorr},  these are arbitrarily well approximated around the origin by the line fields $$-\beta_2 w_{xy} dx^2 +(\beta_1 w_{xx}- \beta_2 w_{yy}) dx dy + \beta_1 w_{xy} dy^2=0.$$ Arguing as in the final part of the proof of Theorem  \ref{localth}, the index of these line fields agrees with the index of the vector field $Z:=(-2\beta_2 w_{xy}, \beta_1 w_{xx} - \beta_2 w_{yy})$. Now, using that $\beta_1\leq \beta_2$ and $D^2 w<0$ in $\R^2-\{(0,0)\}$, we obtain $\esiz Z,\nabla w_y\esde <0$ in $\R^2-\{(0,0)\}$. Hence, this index agrees with the one of $\nabla w_x$ and $\nabla w_y$, which we know is non-positive.

Therefore, we have proved that if $h(x,y)$ is not identically zero around $p\in \cU_0$, then $p$ is an isolated point of $\cU_0$ and the principal line fields of the tensor $\alfa$ have non-positive index at $p$. A standard connectedness argument shows then that either $h=u-u_0$ vanishes around any $p\in \cU_0\subset \Sigma$, what implies that $\Sigma=S_0$ up to a translation, or that the principal line fields of $\alfa$ have on $\Sigma$ a finite number of singularities with non-positive index. Since $\Sigma$ has positive Euler characteristic, this second possibility contradicts the Poincaré-Hopf theorem. Thus, we deduce that $\Sigma=S_0$ up to a translation.
\end{proof}

Observe that the proof of Theorem \ref{th:aniso} only uses the hypothesis that $\Sigma$ is quasiconformal with respect to $S_0$ in an open domain of $\Sigma$ that contains the set $\cU_0$. Thus, Theorem \ref{th:aniso} holds under more general conditions, in the spirit of Remark \ref{rem:1}. We omit the exact statement.

Theorem \ref{th:aniso} generalizes several known uniqueness theorems for immersed spheres modeled by elliptic PDEs, and in particular, it implies the solution by the first two authors in \cite{GM3} of Alexandrov's uniqueness conjecture (\cite{A}) for immersed spheres of prescribed curvature in $\R^3$. One example of such a consequence is given in the result below, where $\kappa_1\geq \kappa_2$ and $\eta$ denote, respectively, the principal curvatures and the unit normal of an immersed oriented surface in $\R^3$.

\begin{corollary}(\cite{GM3})\label{cor:anisoo}
Let $\Sigma_0,\Sigma_1$ be two immersed spheres in $\R^3$ that satisfy a prescribed curvature equation
 \begin{equation}\label{prescrin}
W(\kappa_1,\kappa_2,\eta)=0,
\end{equation} 
where $W$ is $C^1$ in the set $\{(x,y,\nu)\in \R^2\times \S^2: x\geq y\}$, and satisfies the ellipticity condition $W_x W_y>0$ on $W^{-1}(0)$.

Assume that $\Sigma_0$ is an ovaloid. Then $\Sigma_1$ is a translation of $\Sigma_0$.
\end{corollary}
One should note that $W$ is not assumed symmetric in the $x,y$ variables. Equation \eqref{prescrin} contains as particular cases the elliptic Weingarten equation $W(\kappa_1,\kappa_2)=0$, the prescribed mean curvature equation $H=\cH(\eta)$ or the \emph{Minkowski problem equation} $K=\cK(\eta)>0$. The hypothesis in Corollary \ref{cor:anisoo} that $\Sigma_0$ is an ovaloid cannot be removed, see the example in \cite[p.460]{GM3} or \cite[Section 5.1]{GM}. This example also shows that Theorem \ref{th:aniso} does not hold if the convex sphere $S_0$ is only assumed to have non-negative curvature, $K\geq 0$.

\begin{proof}
Take $(x_1^0,x_2^0,\nu)\in W^{-1}(0)$, and consider a small neighborhood of this point where $a\leq W_{x_i}\leq b$ for $i=1,2$ and positive constants $a,b$. Then, for any other point of the form $(x_1,x_2,\nu)\in W^{-1}(0)$ in this neighborhood, we have by the mean value theorem
$$0=W(x_1,x_2,\nu)-W(x_1^0,x_2^0,\nu) = \sum_{i=1}^2 A_ i(x_i-x_i^0),$$ where $A_i=A_i(x_1,x_2,x_1^0,x_2^0,\nu)$, with $a\leq A_i\leq b$. In particular, $$-\frac{b}{a} (x_1-x_1^0)\leq x_2-x_2^0 \leq - \frac{a}{b} (x_1-x_1^0).$$

These inequalities guarantee that if $\Sigma_0,\Sigma_1$ are as in the statement of Corollary \ref{cor:anisoo}, then $S_0:=\Sigma_0$ and $\Sigma:=\Sigma_1$ are in the conditions of Theorem \ref{th:aniso}; see e.g. \eqref{cuasiman}. Thus, the result follows immediately from Theorem \ref{th:aniso}.
\end{proof}
\appendix

\section{Appendix: Examples}

It is well-known that umbilics of CMC surfaces in $\R^3$, and more generally of elliptic Weingarten surfaces of the form $W(H,K)=0$, are isolated and of negative index. We prove next that, even though non-totally umbilical quasi-CMC surfaces only have isolated umbilics (Theorem \ref{localth}), these umbilics can actually have index zero. In particular, this bifurcation from the classical theories suggests that there could exist quasi-CMC tori in $\R^3$ with a finite, non-empty umbilic set.

We thank Giovanni Alessandrini and Albert Clop for helpful conversations regarding the existence of this example.

\begin{lemma}\label{lem:eje1}
Consider the homogeneous polynomial of degree $6$ $$h(x,y)= x y (x^2+y^2) (x^2 +16y^2).$$ Then, the graph $z=h(x,y)$ is a quasiminimal surface in a neighborhood of the origin, and it has at the origin an isolated umbilic of index zero.
\end{lemma}
\begin{proof}
A computation using polar coordinates $(r,\theta)$ in the plane shows that $h_{xx} h_{yy}-h_{xy}^2<0$ holds when $r=1$. In particular, this implies that 
\begin{equation}\label{ap:ho}
(h_{xx} + h_{yy})^2 \leq \mu \{ (h_{xx} - h_{yy})^2 + 4 h_{xy}^2\}
\end{equation}
holds when $r=1$, for some constant $\mu<1$. By homogeneity, these equations hold globally in $\R_*^2 :=\R^2-\{(0,0)\}$. In particular, the graph $z=h(x,y)$ has a unique umbilic, situated at the origin. One can check that $h_{xy}>0$ in $\R_*^2$, and this implies that the vector field $\nabla h_x =(h_{xx},h_{xy})$ has index zero around $(0,0)$. As was explained at the end of the proof of Theorem \ref{localth}, this is equivalent to the fact that the principal line fields of the graph around the origin have index zero, as claimed.

It remains to show that the graph $z=h(x,y)$ is quasiminimal in a neighborhood of the origin, i.e., it satisfies $H^2 \leq \mu (H^2-K)$ for some $\mu<1$. But this property follows directly from \eqref{ap:ho} and the asymptotic expansions $$H^2-K = \frac{1}{4}(h_{xx} - h_{yy})^2 +  h_{xy}^2 + o(r^8), \hspace{0.5cm} H = \frac{1}{2} (h_{xx} + h_{yy}) + o(r^4),$$ what proves the Lemma.
\end{proof}

Let us remark that, by \eqref{ap:ho}, $h(x,y)$ is a solution to a linear, uniformly elliptic equation \eqref{linpde} with bounded (discontinuous) coefficients $a_{ij}$, and with $b_1=b_2=0$, which has at the origin a critical point of index zero. In particular, the origin is not a \emph{geometric critical point} in the sense of \cite{AM}. This contrasts with the local behavior around critical points of solutions to elliptic equations \eqref{linpde} with smooth ($C^1$, Holder, etc.) coefficients, given by Bers, Carleman, Hartman-Wintner and others, since in these cases the index is always negative at critical points.

In the next example we construct a wide family of smooth compact surfaces of genus zero embedded in $\R^3$ that satisfy the Alexandrov inequality \eqref{ecur1}. This proves that the quasi-CMC condition \eqref{cuasi} in Theorem \ref{th:main1} cannot be weakened to its natural limit, given by \eqref{ecur1}.
These examples are similar to a construction in \cite{Mu}.
\begin{figure}[htbp]
    \centering
         \includegraphics[width=4cm]{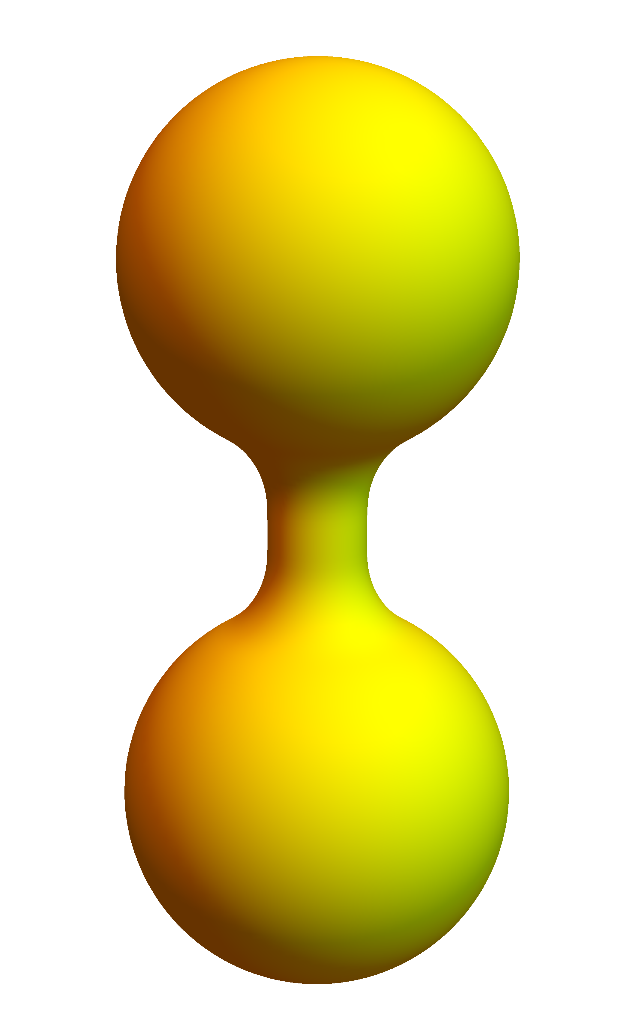}  \hspace{1.5cm} \includegraphics[width=4.4cm]{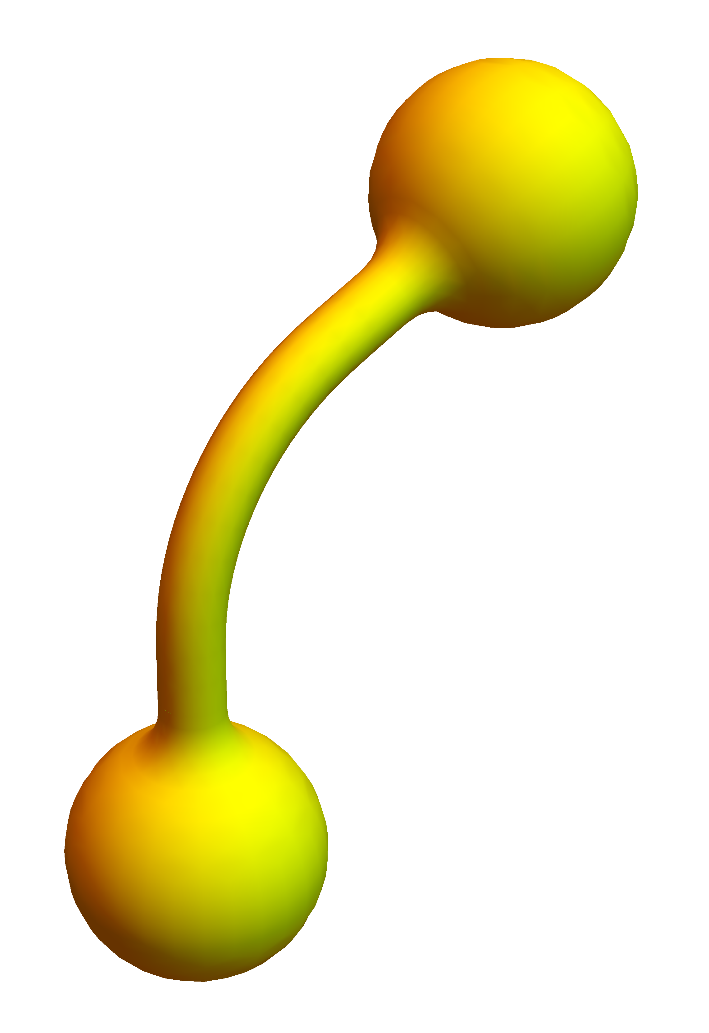}
    \caption{Smooth embedded spheres satisfying the Alexandrov inequality \eqref{ecuno}}
    \label{fig:contra}
\end{figure}

\begin{lemma}\label{lem:ej2}
There exist non-round $C^{\8}$ spheres embedded in $\R^3$ satisfying 
\begin{equation}\label{ecuno}
(\kappa_1-1)(\kappa_2-1)\leq 0, \hspace{0.5cm} \text{and equality holds only if $\kappa_1=\kappa_2=1$}.
\end{equation}
\end{lemma}
\begin{proof}
Take $a,b\in (\frac{\pi}{2},\pi)$, with $a<b<a+\sin(a)$. Let $\kappa\in C^{\8}([0,b])$ satisfy that $\kappa(s)=1$ for all $s\in [0,a]$, $\kappa'(s)<0$ for all $s\in (a,b]$, and 
\begin{equation}\label{incu}
\int_a^b \kappa(s) ds =  \frac{\pi}{2}-a <0.
\end{equation}
Consider the curve $\alfa(s)=(x(s),z(s))$ in $\R^2$ parametrized by arc-length, given by $\alfa(0)=(0,0)$, $\alfa'(0)=(1,0)$, and the prescribed curvature function $\kappa(s)$. Then, $\alfa'(s)=(\cos \theta(s),\sin \theta(s))$ where $\theta'=\kappa$. Note that $\alfa(s)=(\sin s, 1-\cos s)$ for $s\in [0,a]$. In particular, since $x(a)=\sin(a)$ and $\alfa(s)$ is parametrized by arc-length, it follows from $b-a<\sin(a)$ that $x(s)>0$ for all $s\in (0,b]$, independently of our choice of $\kappa(s)$. The condition \eqref{incu} implies that $\alfa'(b)=(0,1)$. Since $\kappa'<0$ in $(a,b)$, \eqref{incu} does not hold for any $b'\in (a,b)$ and thus $\alfa'(s)\neq (0,1)$ if $s\in [0,b)$. Besides, since $\kappa\leq 1$, we have $\theta(s)<\pi$ for all $s$. Also, note that $\theta(s)\in (\pi/2,\pi)$ when $s\in (a,b)$, since \eqref{incu} does not hold for any $b'\in (a,b)$. In particular, $z(s)$ is increasing in $(a,b)$.

Then, we can glue $\alfa$ with its reflection across the $z=z(b)$ line in the $x,z$-plane, and rotate this curve around the $z$-axis to obtain an embedded rotational sphere $\Sigma$ in $\R^3$, which can obviously be constructed with $C^{\8}$ regularity by choosing $\kappa(s)$ adequately. One can think of $\Sigma$ as a rotational sandglass, made of two large spherical caps $S_1,S_2$ of radius $1$, joined by a neck $\cN$; see Figure \ref{fig:contra}, left.

We claim that any such embedded $\Sigma$ satisfies \eqref{ecuno}. First, it is clear that the principal curvature $\kappa(s)$ of $\Sigma$ corresponding to the meridian curves is $1$ in $S_1\cup S_2$, and smaller than $1$ in $\cN$. The principal curvature $\mu(s)$ corresponding to the parallels of $\Sigma$ is also $1$ in $S_1\cup S_2$, but is greater than $1$ in $\cN$ as we explain next. Since $\theta'=\kappa$, $x'=\cos (\theta)$, and $\kappa'<0$ in $(a,b)$, by differentiation  we have $x \kappa - \sin (\theta)<0$ in $(a,b)$. As the principal curvature $\mu(s)$ is given by $\mu=\sin (\theta)/x$, we conclude from the previous equations that $\mu'(s)>0$ in $(a,b)$. Therefore, $\mu(s)>1$ for all $s\in (a,b)$. So, $\Sigma$ satisfies \eqref{ecuno}, and this finishes the proof.

It is interesting to observe that any small normal variation of the sphere $\Sigma$ in a compact region of the interior of the neck $\cN$ also produces a smooth embedded sphere that satisfies \eqref{ecuno}, this time not rotational. So, there is a large class of smooth embedded spheres satisfying \eqref{ecuno}. One can also construct other examples by considering two spherical caps $S_1,S_2$ of radius $1$ with corresponding rotational \emph{half-necks} $\cN_1,\cN_2$ as above, but now joined by a long tube of fixed radius over a smooth regular curve in $\R^3$ (see Figure \ref{fig:contra}, right). We omit the details.
\end{proof}

\def\refname{References}

\vskip 0.2cm

\noindent José A. Gálvez

\noindent Departamento de Geometría y Topología,\\ Universidad de Granada (Spain).

\noindent  e-mail: {\tt jagalvez@ugr.es}

\vskip 0.2cm

\noindent Pablo Mira

\noindent Departamento de Matemática Aplicada y Estadística,\\ Universidad Politécnica de Cartagena (Spain).

\noindent  e-mail: {\tt pablo.mira@upct.es}



\noindent Marcos Paulo Tassi

\noindent Instituto de Ciências Matemáticas e de Computação,\\ Universidade de São Paulo (Brazil).

\noindent  e-mail: {\tt mtassi@dm.ufscar.br}

\noindent J.A. Gálvez was partially supported by MINECO/FEDER Grant no. MTM2016-80313-P, Junta de Andalucia grants no. A-FQM-139-UGR18 and P18-FR-4049. \\[0.2cm] P. Mira  was partially supported by MINECO/FEDER Grant no. MTM2016-80313-P. \\[0.2cm] M.P. Tassi was partially supported by grant no. 2020/03431-6, São Paulo Research
Foundation (FAPESP).

%
%

\end{document}